\documentclass[reqno]{amsart}
\usepackage[utf8]{inputenc} 
\usepackage{amsmath}
\usepackage{bbm}
\usepackage{mathtools}
\usepackage{hyperref}
\usepackage{amssymb}
\usepackage{amsthm}

\theoremstyle{definition}
\newtheorem{definition}{Definition}[section]

\theoremstyle{plain}

\newtheorem{theorem}{Theorem}

\newtheorem{lemma}{Lemma}

\numberwithin{lemma}{section}
\numberwithin{theorem}{section}
\numberwithin{equation}{section}

\DeclareMathOperator{\skmin}{\mathcal{S}_k^{\min}}
\DeclareMathOperator{\sknew}{\mathcal{S}_k^{\text{new}}}
\DeclareMathOperator{\skfull}{\mathcal{S}_k}
\DeclareMathOperator{\gl2}{GL_2}
\DeclareMathOperator{\sl2}{SL_2}
\DeclareMathOperator{\cond}{\mathfrak{f}}

\title{Twist-minimal trace formula for holomorphic cusp forms}

\author{Kieran Child \\ University of Bristol}
\email{kieran.child@bristol.ac.uk}
\address{University of Bristol \\ Beacon House \\ Queens Road \\ Bristol \\ BS8 1QU \\ UK}

\begin{document}

\begin{abstract}
We derive an explicit formula for the trace of an arbitrary Hecke operator on spaces of twist-minimal holomorphic cusp forms with arbitrary level and character, and weight at least 2. We show that this formula provides an efficient way of computing basis elements for newform or cusp form spaces. This work was motivated by the development of a twist-minimal trace formula in the non-holomorphic case by Booker, Lee and Str\"ombergsson, as well as the presentation of a fully generalised trace formula for the holomorphic case by Cohen and Str\"omberg.
\end{abstract}

\maketitle

\section{Introduction}
Modular forms play a central role in much of modern number theory. Their properties have been crucial in establishing the modularity theorem (in turn proving Fermat's last theorem, see \cite{FermatsLastTheorem}), and the Fourier coefficients of modular forms provide explicit data for problems such as the congruent number problem, representations of integers by quadratic forms, and classifications of Galois representations (see \cite{ZagierApplications} for an overview of several applications).

Cusp forms are of particular interest for many reasons. For example, cusp forms give rise to holomorphic $L$-functions (which in turn account for all $L$-functions arising from elliptic curves over $\mathbb{Q}$) and are associated with families of irreducible $\ell$-adic Galois representations following work by Deligne and Serre (see \cite[Section 9.6]{DiamondShurman}).

While there exist simple, explicit expressions for Fourier coefficients of Eisenstein series, such as in \cite{Young}, basis elements for spaces of cusp forms are trickier to compute. In this paper we derive a formula for the trace of Hecke operators on spaces of so-called `twist-minimal' forms (Theorem \ref{MainTheorem}), and demonstrate that this is sufficient to recover the Fourier coefficients of basis elements of cusp form spaces (Theorem \ref{StructuralTheorem}). We give a result enabling the extraction of Fourier coefficients of newforms from twist-minimal representatives (Theorem \ref{thm:twistedcoeffs}).

The value of this paper is twofold. Firstly, the formula is given in an explicit, classical form, and so the terms are readily computable. Secondly, the sieving of the formula to twist-minimal spaces results in much simpler expressions to compute these spaces, and allows us to calculate a smaller number of twist-minimal spaces in order to recover a large number of different cusp form spaces at once. This efficiently provides explicit data for aforementioned problems on representation of integers and Galois representations. The results can also be incorporated into Schaeffer's method for computing weight-1 spaces, given in \cite{Schaeffer}. Further work implementing the results of this paper for the efficient computation of all forms of weight 1 and level at most 10,000 is currently in progress. A full explanation and discussion of the advantages of this method is given at the end of section \ref{Results}.

The idea of using trace formulae to derive information on modular forms dates back to the 1950s. Following the construction of certain non-holomorphic cusp forms, now called Maass cusp forms, by Maass in \cite{Maass}, Selberg sought to demonstrate the infinitude of such forms and provide an asymptotic expression for their count, which he achieved with a trace formula in \cite{Selberg}. He and Eichler continued work in this direction, leading to a trace formula for the action of Hecke operators $T_n$ on spaces of holomorphic cusp forms for the full modular group and square-free $n$ in \cite{Eichler}. Following work by others in \cite{Hijikata} and \cite{Oesterle}, this formula is now available in full generality for traces of any Hecke operator on spaces of any level and character. Although this formula can be used to find Fourier coefficients for basis elements of cusp form spaces (as done in \cite{LMFDBpaper}, with a newform sieve, for the generation of LMFDB data), our formula contains much simpler expressions, and is more efficient for this task in the case that the level $N$ is not square-free.

On the structural side, twists of modular forms were studied in \cite{AtkinLi} and the concept of a twist-minimal form\footnote{In \cite{AtkinLi} this would be termed a `$q$-primitive form for all prime $q$.'} was explored. In \cite{Palm}, Palm presented an adèlic trace formula for twist-minimal forms. This approach was also taken in the classical context in \cite{BookerLeeStrombergsson} for Hecke operators $T_{\pm 1}$ and Maass cusp forms. The pathway between the adèlic and classical contexts is shown in \cite{KnightlyLi}. This motivates us to use the structural idea of restricting to twist-minimal forms in a classical setting for holomorphic cusp forms. The result will be a computable analogue of the adèlic formula in \cite{Palm}.

The structure of the paper is as follows. Section 2 covers the required background theory leading to the presentation of two results. The formula for the trace of Hecke operators on spaces of twist-minimal modular forms is given in Theorem \ref{MainTheorem}, and the extraction of basis elements for twist-minimal, newform or cusp form spaces from this trace formula is given in Theorem \ref{StructuralTheorem}. A formula for the Fourier coefficients of an arbitrary newform given Fourier coefficients of a suitable twist-minimal form is given in Theorem \ref{thm:twistedcoeffs}. The remainder of the paper explains how these theorems are arrived at, with section \ref{StructuralSection} addressing the structural theory, and section \ref{traceformula} deriving the trace formula.
\subsection*{Acknowledgements}
I would like to thank Min Lee and Andrew Booker for their guidance and suggestions with this research.
\section{Twist-minimal trace formula}
\label{Results}
\subsection{Preliminary theory}
For any $N \in \mathbb{N}$, a Dirichlet character $\chi$ is a homomorphism:
\begin{equation}
\chi : \left(\mathbb{Z}/N\mathbb{Z}\right)^* \rightarrow \mathbb{C}^*.
\end{equation}
The domain of $\chi$ is extended to $\mathbb{Z}$ by the canonical mod-$N$ homomorphism, with $\chi(x)=0$ whenever $(N,x) \not = 1$. The number $N$ in this construction is called the level of the character. The conductor of $\chi$, denoted $\cond(\chi)$, is the least $M \in \mathbb{N}$ such that $a \equiv b \pmod{M}$ implies $\chi(a)=\chi(b)$.

If $\cond(\chi)=N$ then we say $\chi$ is primitive. For any $N$, if $\cond(\chi)=1$ then we say $\chi$ is trivial mod $N$. This character will be written as $\mathbbm{1}$ with the level taken from context. Dirichlet characters factor over primes in the following sense. Let $\nu_p(x)$ be the maximum power of $p$ dividing $x$. There exists for all $p\mid N$ a unique $\chi_p$ with level $p^{\nu_p(N)}$ so that:
\begin{equation}
\chi(x)=\prod_{p\mid N}\chi_p(x).
\end{equation}
We use this factorisation to define $\chi_p$ for any prime $p\mid N$. The order of a non-trivial character is the least $n \in \mathbb{N}$ such that $\chi^n=\mathbbm{1}$. This is denoted Ord$(\chi)$. We define twist-minimal characters as in \cite[Definition 1.5]{BookerLeeStrombergsson}.
\begin{definition}\label{tmcdef}
Fix a prime $p$ and $e \in \mathbb{N}$. Let $\chi$ be a Dirichlet character level $p^e$, and fix $s \in \mathbb{Z}_{\ge 0}$ so that $\cond(\chi)=p^s$. We say that $\chi$ is twist-minimal if and only if one of the following holds:
\begin{itemize}
\item $\chi$ is primitive;
\item $p>2$ and $\chi=\mathbbm{1}$;
\item $p>2$ and $\text{Ord}(\chi)=2^{\nu_2(p-1)}$;
\item $p=2$ and $s= \lfloor \frac{e}{2} \rfloor$;
\item $p=2$, $s=2$, $e>3$ and $2\nmid e$;
\item $p=2$, $\chi=\mathbbm{1}$, and $2 \nmid e$ or $e=2$.
\end{itemize}
In general, we define a twist-minimal character $\chi$ as a product $\prod \chi_p$ where all of the $\chi_p$ are twist-minimal.
\end{definition}
This completes the prerequisites on Dirichlet characters, and we move onto modular forms. Let $\mathbb{H}$ be the Poincar\'e upper half plane $\mathbb{H}= \left\{ z \in \mathbb{C} : \Im(z) > 0 \right\}.$ We get an action on $\mathbb{H}$ by M\"obius transformations with notation:
\begin{equation}
\gamma z = \frac{az+c}{cz+d},\;\;\;\;\;\;\; \forall \; \gamma = \begin{pmatrix} a & b \\ c & d \end{pmatrix} \in \gl2(\mathbb{R})^+.
\end{equation}
For any complex-valued function $f$ the weight $k$ slash action of $\gamma = \begin{pmatrix} a & b \\ c & d \end{pmatrix} \in \gl2(\mathbb{R})^+$ on $f$ is defined as:
\begin{equation}
f|_k \gamma(z)=(\text{det}\gamma)^{k/2}(cz+d)^{-k}f(\gamma z).
\end{equation}
For any $N \in \mathbb{N}$, the congruence subgroup $\Gamma_0(N)$ is defined to be the following subgroup of $\sl2(\mathbb{Z})$:
\begin{equation}
\Gamma_0(N)=\left\{ \begin{pmatrix} a & b \\ c & d \end{pmatrix} \in \sl2(\mathbb{Z}): N\mid c \right\}.
\end{equation}
Fix $k, N \in \mathbb{N}$ and fix $\chi$ a Dirichlet character mod $N$. A modular form of weight $k$, level $N$ and character $\chi$ is a complex-valued, holomorphic function $f$, defined on $\mathbb{H}$, which satisfies:
\begin{enumerate}
\item $f|_k\gamma(z)=\chi(d)f(z)$ for all $\gamma \in \Gamma_0(N)$.
\item $\int_{\Gamma_0(N) \backslash \mathbb{H}}|f|^2 \frac{dxdy}{y^2}$ is finite.
\end{enumerate}
Modular forms have Fourier expansions:
\begin{equation}
f(z)=\sum_{n \ge 0}a_n e^{2\pi i n z}.
\end{equation}
Such an expansion completely describes the form, and so it is these $a_n$ coefficients which we seek to compute.

All modular forms of weight $k$, level $N$ and character $\chi$ form a space, denoted $\mathcal{M}_k(N,\chi)$. An important family of operators on this space are the Hecke operators. We follow the definition and notation of \cite[Section 6.5]{IwaniecTopics}. Let $f$ be a modular form of weight $k$ and character $\chi$. For any $n \in \mathbb{N}$ we define the generic Hecke operator\footnote{This definition is more general than in the original literature, where for $p\mid N$ our operator $T_p$ would be written $U_p$.} as:
\begin{equation}
T_n^\chi(f) = \frac{1}{n}\sum_{ad=n}\chi(a)a^k \sum_{b\text{ mod }d}f\left( \frac{az + b}{d}\right).
\end{equation}
Here, $\chi$ is included in the notation for clarity, even though this is determined by $f$. Hecke operators are commutative and multiplicative, so that $T_{nm}^\chi=T_n^\chi T_m^\chi$ whenever $(n,m)=1$. We aim to give an explicit expression for the trace of these operators on subspaces of modular forms called cusp forms, and we show these traces can be used to generate basis elements.

A cusp form is a modular form for which the slash action $f|_k\gamma(z)$ tends to 0 as $\Im(z)\rightarrow \infty$ for all $\gamma \in \Gamma_0(N)$. Cusp forms of weight $k$, level $N$ and character $\chi$ form a subspace of modular forms, denoted $\mathcal{S}_k(N,\chi)$. This subspace is stable under the action of Hecke operators, and is a Hilbert space with respect to the Petersson inner product, defined as:
\begin{equation}
\langle f,g \rangle_N := \int_{\Gamma_0(N)\backslash \mathbb{H}} f(z)\overline{g(z)}y^k \frac{dxdy}{y^2}.
\end{equation}
This inner product is used to decompose the space of cusp forms. Fix $M|N$ and $d|\frac{N}{M}$, then it is verified that for any level $M$ cusp form $f$ the function $g(z)=f(dz)$ is a level $N$ cusp form with the same weight and character as $f$. Forms of level $N$ which arise in this manner from any level $M<N$ are called oldforms, and the orthogonal complement (with respect to the Petersson inner product) of the space spanned by oldforms is called the newform space. It is denoted $\sknew(N,\chi)$, and was studied in detail in \cite{AtkinLehner} and \cite{AtkinLi}. The utility of this decomposition is that one can recover data on cusp form spaces solely from data on newform spaces, and so the scope of any study into cusp forms is reduced.

A cusp form which is an eigenfunction for all Hecke operators is called a Hecke-eigenform. Many results about the corresponding Hecke-eigenvalues are given in \cite{AtkinLi} and \cite{Ogg}. We say that two Hecke-eigenforms in $\skfull(N,\chi)$ are equivalent if they have the same eigenvalues for all but finitely many $T_p$. We thus normalise these forms by setting their first Fourier coefficient to 1. For a normalised Hecke-eigenform $f$, the eigenvalue of $T_n^\chi$ is equal to the $n$-th Fourier coefficient of $f$.

By \cite[Lemma 18]{AtkinLehner} there is a basis for $\sknew(N,\chi)$ consisting of representatives from distinct equivalence classes of simultaneous eigenforms. In the other direction, by \cite[Lemmas 21, 22]{AtkinLehner}, every simultaneous eigenform of a given level $N$ is equivalent to a normalised newform of some level $M\mid N$.

We define the twist of a form $f$ by a Dirichlet character $\psi$ by the following action on its Fourier expansion:
\begin{equation}
f_\psi(z)=\sum_{n\ge 0} \psi(n) a_n  e^{2\pi i nz}.
\end{equation}
We immediately get the following interaction between twists and hecke operators.
\begin{lemma}
\label{hecketwistinteraction}
Let $f$ be a Hecke-eigenform of weight $k$, level $N$ and character $\chi$. Let $\psi$ be a character of level $M$. If $(M,n)=1$ then $T_n^{\chi \psi^2} (f_\psi) = \psi(n)(T_n^{\chi}(f))_{\psi}$.
\end{lemma}
This relation, first stated in \cite{AtkinLi}, is checked by expanding out the definitions of the operators involved. We also verify that $f_\psi$ is a cusp form with character $\chi \psi^2$. This means that the twist of any Hecke-eigenform is equivalent to a newform basis element. We can use twists to pass between Hecke-eigenforms of different levels and characters. 

We define a twist-minimal form as a normalised newform which is not equivalent to any form arising as the twist of a newform from a lower level. Twist-minimal forms of weight $k$, level $N$ and character $\chi$ again span a Hilbert space, stable under Hecke operators, which we denote $\skmin(N,\chi)$. As in the newform decomposition, the immediate utility here is from further restriction to the scope of study.

\subsection{Results}
It will be shown in Theorem \ref{StructuralTheorem} that basis elements for subspaces of cusp forms (given as Fourier expansions) can be computed from the trace of Hecke operators acting on spaces of twist-minimal forms with twist-minimal character. Our first result is an explicit formula for this trace.
\begin{theorem}
\label{MainTheorem}
Fix a weight $k \ge 2$, a level $N \in \mathbb{N}$ and a twist-minimal character $\chi$ of level $N$. Let $\chi$ factor over the primes as $\chi=\prod_{p\mid N}\chi_p$, and denote the conductor $\cond(\chi)$. Let $\text{Tr}T_n^\chi|\skmin(N,\chi)$ denote the trace of the $n$-th Hecke operator acting on the space of twist-minimal forms of level $N$, character $\chi$ and weight $k$. If $\gcd((N/\cond(\chi))^2,n^2,N)$ is not square-free, or if $\chi(-1)\not = (-1)^k$, then $\text{Tr}T_n^\chi|\skmin(N,\chi)=0$. Otherwise, we have:
\begin{equation}
\text{Tr}T_n^\chi|\skmin(N,\chi)=C_1-C_2-C_3+C_4,
\end{equation}
where the terms are defined as follows:
\begin{equation}
\begin{split}
C_1&=\frac{n^{\frac{k}{2}-1}(k-1)\chi_{\cond}(\sqrt{n})}{12}\prod_{p\mid N}\begin{cases}
p^e+p^{e-1}& \text{if }s=e\\
\frac{\phi (\lceil p^{e-2}\rceil)(p-1)}{1+\underset{2\mid e,p>2}{\delta}}(1+\underset{e>1}{\delta}p+\underset{e=2}{\delta}(2s-2))&\text{if }s<e\\
\end{cases}\\
C_2&=\sum_{t^2<4n} \frac{\rho^{k-1}-\overline{\rho}^{k-1}}{\rho-\overline{\rho}}\frac{h(d)}{w(d)}\prod_{\substack{p\mid \ell \\ p\nmid N}}S_p(1,\mathbbm{1},t,n)\prod_{p\mid N}S_p^{\min}(p^e,\chi_p,t,n)\\
C_3&=\sum_{\substack{d\mid n \\ d \le \sqrt{n}}}' d^{k-1}\prod_{p\mid N}\begin{cases}
\frac{\sqrt{2}^e (\chi_p(d)+\chi_p(n/d))}{8}(1-\underset{\gamma=\frac{e}{2}-1}{\delta}2) & \text{if }p=2, 2\mid e, \gamma \ge \frac{e}{2}-1 \ge s, 2\nmid n, e>2\\
\chi_p(d)+\chi_p(n/d) & \text{if }s=e\\
0 & \text{ otherwise}\\
\end{cases}\\
C_4&=\underset{k=2,\chi=\mathbbm{1}}{\delta} \mu(N) \prod_{\substack{p\mid n \\ p\nmid N}}\sigma(p^{\nu_p(n)})\\
\end{split}
\end{equation}
In all terms, $e=\nu_p(N)$ and $s=\nu_p(\cond(\chi))$ for a given prime $p$, and $\delta$ is the characteristic function of the condition underneath it.

We give details of each term in this formula. In $C_1$, $\chi_{\cond}$ is the primitive character which induces $\chi$, and $\chi_{\cond}(\sqrt{n})$ is 0 if $n$ is not square. In $C_3$ we define $\gamma=\nu_p(n/d-d)$ and the dash on the sum indicates that we have an additional factor of $\frac{1}{2}$ if $d=\sqrt{n}$.

In $C_2$, $\rho$ is a root of the polynomial $x^2-tx+n$. We define $d$ as the fundamental discriminant of $t^2-4n$, and $\ell \in \mathbb{N}$ such that $t^2-4n=d\ell^2$. The functions $h$ and $w$ are the class number and roots of unity of a quadratic extension with that fundamental discriminant. For the $S_p^{\min}$ terms, set $\gamma = \nu_p(t^2-4n)$, and let $\left(\frac{\cdot}{p}\right)$ denote the Kronecker symbol, then if $p>2$ and $(p,n)=1$ we have:
\begin{equation}
S_p^{\min}(p^e,\chi,t,n)=\begin{cases}\underset{\substack{e=1 \\ \text{ or } \left(\frac{n}{p}\right)=1}}{\delta} \left(1-\left(\frac{d}{p}\right)\right)\frac{p^{e-3}}{(2,e)}\chi(t/2)\\ \cdot \left(\underset{e>2}{\delta}+p\left(\underset{e=2}{\delta}(1-2s)+\underset{\substack{2\mid e \\ \gamma=e-2}}{\delta}-\underset{\gamma \ge e-1}{\delta}p\right)\right) & \text{if }s<e,\gamma \ge e-2\\
\chi(t/2)\left(2p^{\nu_p(\ell)}+\left(1-\left(\frac{d}{p}\right)\right)\frac{2p^{\nu_p(\ell)}-p^e-p^{e-1}}{p-1}\right) & \text{if }s=e,\gamma \ge 2e-1 \\
p^{\nu_p(\ell)}(\chi(\frac{t+u}{2})+\chi(\frac{t-u}{2})) & \text{if }s=e,\gamma<2e-1, \left(\frac{d}{p}\right)=1\\ 
0 & \text{else}
\end{cases}
\end{equation}
If $p=2$, $s<e$, and $(p,n)=1$ we have:
\begin{equation}
S_2^{\min}(2^e,\chi,t,n)=\left(1-\left(\frac{d}{2}\right)\right)\lceil 2^{e-3}\rceil \begin{cases}
-3\chi\left(\frac{t}{2}\right) & \text{if }\gamma > e, e \ge 3\\
\chi\left(\frac{t}{2}\right)((-1)^e+2) & \text{if }\gamma=e,s=\lfloor \frac{e}{2}\rfloor, e \ge 4 \\
\chi\left(\frac{t}{2}\right)(1-2(-1)^d) & \text{if }\gamma=e-1,2\nmid e, s=\lfloor \frac{e}{2} \rfloor, e \ge 4\\
2(-1)^{d}-1 & \text{if }\gamma \in \{e,e-1\}, s<\lfloor \frac{e}{2} \rfloor, e \ge 3\\
\underset{\substack{e=2\\ \gamma=0}}{\delta}\frac{3}{2}-1 & \text{if }e \in \{1,2\}\\
0 & \text{ otherwise}
\end{cases}
\end{equation}
If $p=2$, $e=s$ and $(p,n)=1$ we have:
\begin{equation}
S_2^{\min}(2^e,\chi,t,n)=\begin{cases}
(1-\underset{\gamma=2e}{\delta}2)\chi(t/2) \\
\cdot\left((2^{\lfloor \frac{\gamma}{2}\rfloor+1}-3\cdot2^{e-1})(1-\left(\frac{d}{2}\right))+\underset{2\nmid d}{\delta}2^{\nu_2(\ell)+1}\right)& \text{if }\gamma \ge 2e\\ \\
2^{\nu_2(\ell)}(\chi\left(\frac{t+u}{2}\right)+\chi\left(\frac{t-u}{2}\right)) & \text{if }\gamma<2e-1,\left(\frac{d}{2}\right)=1\\ \\
0 & \text{otherwise}
\end{cases}
\end{equation}
and if $(p,n)=p$ we have:
\begin{equation}
S_p^{\min}(p^e,\chi,t,n)=\begin{cases}
\left(\frac{d}{p}\right)-1 & \text{if }\gamma>s=0\\
\chi\left(\frac{t-u}{2}\right)+\chi\left(\frac{t+u}{2}\right) & \text{if }s=e, \gamma=0 \\
0 & \text{ otherwise}
\end{cases}
\end{equation}
In all cases, $u\equiv \ell\sqrt{d}$ is defined mod $p^e$ if $p>2$ and $p^{e+2}$ if $p=2$. Finally, we have:
\begin{equation}
S_p(1,\mathbbm{1},t,n)=p^{\nu_p(\ell)}+\left(1-\left(\frac{d}{p}\right)\right)\frac{p^{\nu_p(\ell)}-1}{p-1}
\end{equation}
\end{theorem}
\begin{definition}
Fix a weight $k$, level $N$ and character $\chi$. The \emph{trace form} of $\skmin(N,\chi)$, denoted $\mathcal{T}_{N,\chi}$, is given by the following series on $z \in \mathbb{H}$:
\begin{equation}
\mathcal{T}_{N,\chi}(z)=\sum_{n \ge 1} \text{Tr}T_n^\chi|\skmin(N,\chi) e^{2 \pi i n z}.
\end{equation}
\end{definition}
We can extract basis elements for $\skmin$ from its trace form by the following lemma.
\begin{lemma}
Let $\mathcal{T}_{N,\chi}$ be the trace form for the twist-minimal space $\skmin(N,\chi)$. Then all $T_n\mathcal{T}_{N,\chi}$ are in $\skmin(N,\chi)$. Let $M$ be a matrix with each colum $M_n$ given by the coefficients of $T_n \mathcal{T}_{N,\chi}$. If $M$ has at least the number of columns and rows equal to the Sturm bound of the space then $M$ is a basis for the space. i.e. the rank of $M$ equals the dimension of the space.
\end{lemma}
\begin{proof}
Let $f_i$ be a normalised Hecke eigenform basis for $\skmin(N,\chi)$, where the $i$-th element has $n$-th Fourier coefficient $a_{i,n}$. Then we have:
\begin{equation}
T_n\mathcal{T}_{N,\chi}=\sum_i a_{i,n} f_i \in \skmin(N,\chi).
\end{equation}
Further, if we let $B$ denote the matrix with columns given by $f_i$ coefficients, then $M=BB^{\top}$ and so the ranks of $B$ and $M$ are equal.
\end{proof}
\begin{definition}
\label{definitiontwistpair}
Fix a level $N$ and a character $\chi$. A `twist pair', denoted $\langle M, \psi \rangle$, is a tuple of level $M$ and character $\psi$ satisfying at least one of the following for every prime $p$:
\begin{itemize}
\item $\nu_p(M)=\nu_p(N)$ and $\psi=\mathbbm{1}$
\item $2 \nmid p$, $2\mid \nu_p(N)$, $\nu_p(\mathfrak{f}(\chi))<\nu_p(N)$, $\psi_p \not = \chi_p$ and $\nu_p(M)=\nu_p(\cond(\psi))=\frac{\nu_p(N)}{2}$
\item $2\nmid p$, $\nu_p(N)=2$, $\chi_p=\mathbbm{1}$, $\nu_p(M)=0$ and $\psi_p=\left(\frac{\cdot}{p}\right)$
\end{itemize}
\end{definition}
Two twist pairs $\langle M_1, \psi_1 \rangle$ and $\langle M_2, \psi_2 \rangle$ are considered equivalent if $M_1 = M_2$ and $\psi_2 = \overline{\chi \psi_1}$. Basis elements for spaces of cusp forms can subsequently be generated as  using the following result:
\begin{theorem}
\label{StructuralTheorem}
Fix a level $N$ and weight $k$, and let $\chi$ be any (not necessarily twist-minimal) character of level $N$. The space $\skfull(N,\chi)$ has a basis of cusp forms given by:
\begin{equation}
\left(T_m^{\chi \psi^2} \mathcal{T}_{M,\chi \psi^2}\right)_{\overline{\psi}}(dz)
\end{equation}
where $m \in \mathbb{N}$, $\chi \psi^2$ is twist-minimal level $M$, and setting $M'=\max(M,\cond(\psi)\cond(\chi\psi))$ we have $M'\mid N$ and $d\mid \frac{N}{M'}$. Fixing $d=1$ in the above gives a basis for $\sknew(N,\chi)$, and fixing $\psi=\mathbbm{1}$ gives a basis for $\skmin(N,\chi)$, provided $\chi$ is twist-minimal.

If $\chi$ is twist-minimal, then basis elements for any of these spaces are given by letting $M$ and $\psi$ run over all twist pairs for level $N$ and character $\chi$. If $\chi$ is not twist-minimal, then we can generate basis elements from bijection with a twist-minimal space. Further, for any of these spaces, there exists a subset of $m \in \mathbb{N}$ with $m$ at most the Sturm bound of the space such that the resultant basis matrix has zero nullity.\end{theorem}

Finally, we extend Lemma \ref{hecketwistinteraction} to a result giving the Fourier expansion of a newform whenever an appropriate twist-minimal form is known.

\begin{theorem}\label{thm:twistedcoeffs}
Let $f\in\skmin(N,\chi)$ be a twist-minimal form with Fourier coefficients $a_n$. Let $\psi$ be a primitive Dirichlet character, and
let $\psi'$ be the primitive character that induces $\chi\psi$.
Then $f_\psi$ is equivalent to a newform in $\sknew(M,\chi\psi^2)$ with Fourier coefficients $b_n$, where $M=\text{LCM}(N,\cond(\psi)\cond(\psi'))$ and:
\begin{equation}\label{eq:twistedcoeffs}
b_p=
\begin{cases}
a_p\psi(p)
&\text{if }\psi_p \not = \overline{\chi_p}
\text{ or }p\nmid\cond(\psi),\\
\overline{a_p}\psi'(p)&\text{otherwise}
\end{cases}
\end{equation}
for all primes $p$.
\end{theorem}
\begin{proof}
First, let $Q$ be the product of primes $p$ such that $\psi_p = \overline{\chi_p}$, and let $\psi_Q$ be the product of the respective $\psi_p$. Twisting by $\psi_Q$ is equivalent to applying the Atkin and Lehner $W_Q$ operator. The twist $f_{\psi_Q}$ is equivalent to a newform of level $N$ with coefficients as in (\ref{eq:twistedcoeffs}) by \cite[(1.1)]{AtkinLi}.

Denote this newform $g$, and note that $g$ is itself twist-minimal. This means we can apply Lemma \ref{twistresultlemma} to $g_{\psi_{\frac{N}{Q}}}$, whereupon the result follows.
\end{proof}

These results give us an efficient way of computing basis elements of cusp form spaces. The most computationally expensive term in Theorem \ref{MainTheorem} is $C_2$ in which the multiplicative function is a simple case statement, and the relevant class numbers can be pre-calculated and stored in a table. This formula is thus simpler to compute than the respective formula for full cusp form spaces, given in (\ref{traceformulaequation}).

Further, if we let $d_{\text{new}}$ denote the dimension of a newform space $\mathcal{S}^{\text{new}}$, then the complexity of generating a full-rank matrix of leading coefficients of basis elements from the trace formula is $O(d_{\text{new}}^3)$. On the other hand, letting $d_{\min}$ denote the dimension of a particular twist-minimal space $\mathcal{S}^{\min}$ which twists into $\mathcal{S}^{\text{new}}$ then the complexity of computing the same number of coefficients for just these basis elements is $O\left((d_{\min}d_{\text{new}})^{\frac{3}{2}}\right)$. As a result, as well as the individual terms being easier to compute, the overall complexity for computing spaces of newforms (or, in turn, full cusp form spaces) is reduced by recovering these spaces from twist-minimal spaces.

The restriction that the weight be at least 2 means that weight 1 forms are missed in this approach.\footnote{Weight 1 forms are also missed in the modular symbols approach outlined in \cite{Cremona}} This is a meaningful issue as, since the proof of Serre's modularity conjecture in \cite{KhareWintenberger}, these forms are known to categorise all odd, two-dimensional, irreducible Galois representations over finite fields. The Fourier coefficients of these forms give the traces of Frobenius elements in the associated representations.

The computation presented in \cite{Schaeffer}, however, involves first computing weight 2 basis elements, before manipulating these in a manner to find weight 1 forms. In future work, we will adapt this approach, utilising the benefits of the twist-minimal weight 2 computation presented here to perform more efficient computations of weight 1 forms.
\section{Structural theory}
\label{StructuralSection}
We study the structural theory of cusp forms, which leads to the proof of Theorem \ref{StructuralTheorem}. In \cite[Theorem 5]{AtkinLehner} a decomposition of cusp form spaces is given in terms of newforms and lifts of newforms. To decompose these newform spaces in terms of twist-minimal spaces we draw on previous results on twists of cusp forms.
\begin{lemma}
\label{twistresultlemma}
Fix some $f \in \skmin(N,\chi)$ with $\chi$ twist-minimal, and let $\psi$ be any Dirichlet character, then we deduce the following about $f_\psi$:
\begin{enumerate}
\item $f_\psi$ is a newform if and only if $\psi_p \not = \overline{\psi_p}$ for all $p\mid \cond(\psi)$
\item If $f_\psi$ is a newform, then it has level $M$ and character $\chi\psi^2$ where, for any prime $p$, we have $\nu_p(M)=\max(\nu_p(N),\nu_p(\cond(\psi))+\nu_p(\cond(\chi\psi)))$
\end{enumerate}
\end{lemma}
These come from applying the results in \cite{AtkinLi} to the specific case of twist-minimal forms with twist-minimal character. The multiplicativity present in both the expression for the level of $f_\psi$ and the definition of twist-minimal characters will allow us to decompose newform spaces multiplicatively.

Fix a level $N$. For any $\chi$, there exists a character $\psi$ so that $\chi\psi^2$ is twist-minimal, and twisting by $\psi$ gives a bijection between $\skmin(N,\chi)$ and $\skmin(N,\chi \psi^2)$. This is shown in \cite[Lemma 1.6]{BookerLeeStrombergsson}, with a construction of the required $\psi$. As a result, we need only consider twist-minimal spaces with twist-minimal character, and need study only the decomposition of newform spaces with twist-minimal character.
\subsection{Explicit decomposition of a specified newform space}
In this section, we will identify which twist-minimal spaces are required to obtain all basis elements of any given newform space with twist-minimal character.

Because of Lemma \ref{twistresultlemma}, and the definition of twist-minimal characters, we need only study spaces where $N$ (and therefore $\cond(\chi)$) is a power of some prime.
\begin{lemma}
\label{twistfrom1}
Let $p$ be an odd prime and $\chi$ be a twist-minimal character level $p^f$, then there exists a non-trivial twist from $\skmin(p^e,\mathbbm{1})$ to $\sknew(p^f,\chi)$ if and only if $\chi=\mathbbm{1}$, $f=2$ and $e \in \{0,1\}$
\end{lemma}
\begin{proof}
Twisting $\skmin(p^e,\mathbbm{1})$ by a character $\psi$ gives forms with level $\cond(\chi)^2$ and character $\chi=\psi^2$. The only possibility $\psi$ to be non-trivial while $\psi^2$ is twist-minimal is if Ord$(\psi)=2$, and so $f=2$, whereupon $e<f$ implies $e \in \{0,1\}$.
\end{proof}
\begin{lemma}
Let $p$ be an odd prime, and $\chi_2$ a twist-minimal character with $\cond(\chi_2)=p^f$, then there does not exist any twist from $\skmin(p^e,\chi_1)$ to $\sknew(p^f,\chi_2)$ with $e<f$ and $\chi_1$ twist-minimal.
\end{lemma}
\begin{proof}
Assume for contradiction that $\psi$ is a character providing such a twist from $\skmin(p^e,\chi_1)$ to $\sknew(p^f,\chi_2)$. If $\cond(\chi_1 \psi) < \cond(\psi)$ then $\cond(\chi_1)=\cond(\psi)$ and $\cond(\chi_1\psi^2)=\cond(\psi)$. This means $e=s$ which is a contradiction.

On the other hand if we have $\cond(\chi_1 \psi) \ge \cond(\psi)$ then $\cond(\chi_1)\cond(\chi_1 \psi)>\cond(\chi \psi^2)$. This means, by Lemma \ref{twistresultlemma} part 2, that the level $p^f>\cond(\chi_1 \psi^2)$, which contradicts $p^f=\cond(\chi_2)$.
\end{proof}
On the other hand, when the conductor of the character of the higher level space is less than the level (written as $\cond(\chi_2)<p^f$) then we have an abundance of twists into this space.
\begin{lemma}
\label{alltwists}
Let $p$ be an odd prime, $f$ an even integer, and $\chi_2$ a twist-minimal character level $p^f$ with $\cond(\chi_2)<p^f$. Any $\psi$ with $\cond(\psi)=p^{\frac{f}{2}}$, excepting $\psi=\chi_2$, twists from $\skmin(p^{\frac{f}{2}},\chi_1)$ into $\sknew(p^f,\chi_2)$ where $\chi_1=\chi_2\overline{\psi}^2$. This, along with lemma \ref{twistfrom1}, characterises all twists from lower levels.
\end{lemma}
\begin{proof}
This comes from applying Lemma \ref{twistresultlemma} level calculations on all possible twists.
\end{proof}
Having covered all twists when $p$ is odd, we now study $p=2$.
\begin{lemma}
Let $\chi_1$ and $\chi_2$ be twist-minimal characters for $2^e$ and $2^f$ respectively with $f>e$. There does not exist any twist from $\skmin(2^e,\chi_1)$ to $\sknew(2^f,\chi_2)$.
\end{lemma}
\begin{proof}
Assume for contradiction that $\psi$ is such a character. We break the proof down into three cases based on $\cond(\psi)$ and $\cond(\chi_1)$. If $\cond(\psi)>\cond(\chi_1)$ then by level calculation we have $p^f=\cond(\psi)^2$. But the only twist-minimal $\chi_2$ with $2\mid f$ is when $\cond(\chi_2)=p^{\frac{f}{2}}=\cond(\psi)$. However, we also have $\cond(\chi_2)=\cond(\chi_1 \psi^2) \le \text{max}(\cond(\chi_1),\cond(\psi^2)) < \cond(\psi)$, which is a contradiction.

Next, suppose $\cond(\psi)<\cond(\chi_1)$ so that $\cond(\chi_2)=\cond(\chi_1)$ and $p^f = \cond(\psi)\cond(\chi_1)<p^{2e}$. As $\psi \not = \mathbbm{1}$, we must have $\cond(\chi_1)>4$. If $\cond(\chi_1)=p^e$ then there exist no twist-minimal characters for $e<f<2e$. If $\cond(\chi_1)=p^{\lfloor e/2 \rfloor}$ then the only possibility for $f$ is $e+1$, which would mean $\cond(\psi)=2$ which is impossible.

Finally, suppose $\cond(\psi)=\cond(\chi_1)$, so that $\cond(\chi_2)=\cond(\chi_1 \psi^2) = \cond(\chi_1)$ and $p^f=\cond(\psi)\cond(\chi_1 \psi)<\cond(\psi)^2=\cond(\chi_1)^2$. This is the same as the last case except for the fact that we can no longer rule out $\cond(\chi_1)=4$ when $\cond(\psi)=4$. However, when this is the case we get $\cond(\psi \chi_1)=1$ so that $p^f=p^e$, which contradicts $f>e$.
\end{proof}
We can summarise the above lemmas to give a decomposition of newform spaces with twist-minimal character into twists from twist-minimal spaces. Let a twist pair be defined as in Definition \ref{definitiontwistpair}, then we have:
\begin{equation}
\label{notdisjointunion}
\sknew(N,\chi)= \bigcup_{\langle M, \psi \rangle }\skmin(M,\chi \psi^2)_{\overline{\psi}},
\end{equation}
with the union being taken over all valid twist pairs for level $N$ and character $\chi$. Here, the twist notation means that all forms in the space have been twisted by the subscript character.

This is sufficient for Theorem \ref{StructuralTheorem} but in order to use this decomposition in computations we need it to be disjoint.
\subsection{Removing twist-equivalent forms}
\begin{definition}
Let $\skmin(N_1,\chi_2\psi_1^2)$ and $\skmin(N_2,\chi_2\psi_2^2)$ be two spaces in the decomposition given in (\ref{notdisjointunion}). We say that $f \in \skmin(N_1,\chi_2\psi_1^2)$ and $g \in \skmin(N_2,\chi_2\psi_2^2)$ are \emph{twist-equivalent} if $f_{\overline{\psi_1}} = g_{\overline{\psi_2}}$.
\end{definition}
Twist-equivalent forms are thus twist-minimal forms which give rise to the same newform. We see immediately that two forms are twist-equivalent only if $N_1=N_2$, and the following lemma lets us consider twist-equivalence as a condition on twist pairs.
\begin{lemma}
\label{lemmapairequivalence}
Let $\skmin(p^e,\chi_2\psi_1^2)$ and $\skmin(p^e, \chi_2\psi_2^2)$ be two spaces in the decomposition given in (\ref{notdisjointunion}). For any $f \in \skmin(p^e,\chi_2\psi_1^2)$ there exists some twist-equivalent $g \in \skmin(p^e,\chi_2\psi_2^2)$ if and only if we have either $\psi_2=\overline{\chi_2\psi_1}$ or $\psi_2=\psi_1$.
\end{lemma}
\begin{proof}
We know from Lemma \ref{alltwists} that if $\psi_1 \not = \mathbbm{1}$ then $\cond(\chi_2\psi_1^2)=\cond(\chi_2\psi_2^2)=p^e$. Assume that either $\psi_2=\overline{\chi_2 \psi_1 }$ or $\psi_2=\psi_1$, then setting $g=f_{\overline{\psi_1}\psi_2}$ we get $g_{\overline{\psi_2}}=f_{\overline{\psi_1}}$ and the level of $g$ is $p^e$, hence $g \in \skmin(p^e,\chi_2\psi_2^2)$.

In the other direction, assume that $f_{\overline{\psi_1}}=g_{\overline{\psi_2}}$. Then twisting $f$ by $\overline{\psi_1}\psi_2$ gives the level of $g$ as $\cond(\overline{\psi_1}\psi_2)\cond(\chi_2\psi_1\psi_2)$. This only equals $p^e$ if $\psi_2=\overline{\chi_2\psi_1}$ or $\psi_2=\psi_1$.
\end{proof}
In general for level $N$, we see that two forms are twist-equivalent if and only if this condition holds for all $p\mid N$. This allows us to write down a disjoint decomposition. Let $\langle M,\psi\rangle$ be as before, and let $\sim$ denote the relation by this lemma giving rise to twist-equivalence, then we have the decomposition:
\begin{equation}
\label{disjointdecomposition}
\sknew(N,\chi)= \bigoplus_{\langle M, \psi \rangle / \sim}\skmin(M,\chi \psi^2)_{\overline{\psi}}
\end{equation}
We now quantify the equivalence classes. This is also done multiplicatively, and we can find the number of elements in an equivalence class based on which case we are in for each $p\mid N$ in the definition of twist pairs (Definition \ref{definitiontwistpair}).
\begin{lemma}
\label{lemmaequivalencesize}
Fix a level $N$ and Dirichlet character $\chi$. Let $\langle M,\psi\rangle$ be some twist pair in the decomposition of $\sknew(N,\chi)$. We define a function $k\left( \frac{N}{M}, \chi, \psi\right)$, as the number of primes $p\mid \frac{N}{M}$ such that either $\psi_p \not = \left(\frac{\cdot}{p}\right)$ or $p\mid \mid \cond(\chi)$. Then the size of the equivalence class of $\langle M, \psi \rangle$ in (\ref{disjointdecomposition}) is $2^{k\left(\frac{N}{M},\chi,\psi\right)}$. 
\end{lemma}
\begin{proof}
By following Lemma \ref{lemmapairequivalence} we see that if $\psi_p = \mathbbm{1}$ or $\chi_p=\mathbbm{1}$ and $\psi_p=\left(\frac{\cdot}{p}\right)$ then the pair is in a singleton equivalence class. For all other cases, the lemma gives us a second, distinct twist pair. The result then follows from multiplicativity.
\end{proof}
\section{Deriving the trace formula}
\label{traceformula}
In \cite[Theorem 12.4.11]{CohenStromberg} we see the following explicit formula for the trace of a Hecke operator on an arbitrary full cusp form space. The specific details of functions and variables are as in Theorem \ref{MainTheorem}.

Fix a level $N$, character $\chi$ and weight $k \ge 2$. If $\chi(-1) \not = (-1)^k$ then $\text{Tr}T_n^\chi|\skfull(N,\chi)=0$. Otherwise we have:
\begin{equation}
\text{Tr}T_n^\chi|\skfull(N,\chi)=A_1-A_2-A_3+A_4,
\end{equation}
where the individual terms are given by:
\begin{equation}
\label{traceformulaequation}
\begin{split}
A_1& = n^{k/2-1}\chi(\sqrt{n})\frac{k-1}{12}N\prod_{p\mid N}\left(1+\frac{1}{p}\right), \\
A_2& = \sum_{\substack{t \in \mathbb{Z}\\ t^2 < 4n}}\frac{\rho^{k-1}-\overline{\rho}^{k-1}}{\rho-\overline{\rho}} \sum_{f^2\mid (t^2-4n)}\frac{h((t^2-4n)/f^2)}{w((t^2-4n)/f^2)} \prod_{\substack{p\mid N \\ p \nmid \frac{N}{(N,f)}}}\mu(p,N,(f,N),\chi),\\
A_3& = \sum_{\substack{d\mid n \\ d \le \sqrt{n}}}' d^{k-1} \sum_{\substack{c\mid N \\ (c,N/c)\mid (\frac{N}{\cond(\chi)},n/d-d)}} \phi((c,N/c))\chi(x_1),\\
A_4& = \sum_{\substack{t\mid n \\ (n/t,N)=1}}t,\\
\mu(p,N,g,\chi)&=g\cdot \left( 1 + \frac{1}{p}\right) \sum_{\substack{x\text{ mod }N \\ x^2-tx+n \equiv 0 \text{ (mod }Ng)}} \chi(x).\\
\end{split}
\end{equation}
In addition to the details from Theorem \ref{MainTheorem} , the variable $x_1$ is defined by the congruences $x_1 \equiv d$ (mod $c$) and $x_1 \equiv n/d$ (mod $N/c$). These congruences are seen to uniquely fix $x_1$ modulo $\cond(\chi)$, and so the term $\chi(x_1)$ in $A_3$ is well-defined.

From the structural theory, we write the cusp form space trace formula in terms of newform spaces, and subsequently in terms of twist-minimal spaces. We then invert these formulae to get a formula for the twist-minimal space in terms of full cusp form spaces.

The newform trace is written as a sum over the set:
\begin{equation}
\label{mathcalp}
\mathcal{P}(N,\chi,n)=\left\{ x  : \text{for all }p\mid x\text{ we have } p\mid \mid x, p\mid \mid N, \chi_p=\mathbbm{1}\text{ and } p^2\mid n \right\}.
\end{equation}
From now on, the notation $\mathcal{P}$ will always refer to this set.
\begin{lemma}
\label{newformequation} For any $m \in \mathbb{N}$, define $\beta_m$ to be a multiplicative function, defined on powers of primes as:
\begin{equation}
\label{eq:betafunc}
\beta_m(p^a) = \begin{cases}
1 & \text{if }a=0\\
\underset{p\mid m}{\delta}-2 & \text{if }a=1\\
1-\underset{p\mid m}{\delta} & \text{if }a=2\\
0 & \text{if }a \ge 3.\\
\end{cases}
\end{equation}
If gcd$((N/\cond(\chi))^2,n^2,N)$ isn't square-free, then $\text{Tr}T_n^\chi|\sknew(N,\chi)=0$. Otherwise, we have:
\begin{equation}
\text{Tr}T_n^\chi|\sknew(N,\chi) = \sum_{d \in \mathcal{P}(N,\chi,n)}\chi(d) d^{k-1}\sum_{\substack{M\mid N/d\\ \cond(\chi)\mid M}}   \beta_{n/d^2}(N/dM)\text{Tr}T_{n/d^2}^\chi|\skfull(M,\chi)
\end{equation}
\end{lemma}
\begin{proof}
In \cite[Theorem 13.5.7]{CohenStromberg} we get the equation:
\begin{equation}
\text{Tr}T_n^\chi|\sknew(N,\chi) = \sum_{\substack{M\mid N\\ \cond(\chi)\mid M}} \sum_{\substack{d\mid (\frac{M}{\cond(\chi)},N_1) \\ d^2 \mid n}} \chi(d) d^{k-1} \beta_{n/d^2}(N/M)\text{Tr}T_{n/d^2}^\chi|\skfull(M/d,\chi),
\end{equation}
where $N_1$ is the squarefree component of $N$ (the product of primes $p\mid \mid N$). In \cite{AtkinLi} we see that if $p\mid \left(n,\frac{N}{\cond(\chi)}\right)$ with $p^2\mid N$ then the $p$-th Hecke eigenvalue for any form in $\sknew(N,\chi)$ is 0, and so $\text{Tr}T_n^\chi|\sknew(N,\chi)=0$. When this is not the case for any $p$, the lemma follows from changing the order of summation.
\end{proof}
From Lemma \ref{lemmaequivalencesize} we get an equation giving the trace on any newform space in terms of traces on twist-minimal spaces. Let the $k$ function, and $\langle M,\psi \rangle$ notation be as in that lemma, then we get:
\begin{equation}
\text{Tr}T_n^\chi|\sknew(N,\chi) = \sum_{\langle M, \psi \rangle}2^{-k\left(\frac{N}{M} \chi, \psi \right)}\overline{\psi(n)}\text{Tr}T_n^{\chi \psi^2}|\skmin(M,\chi \psi^2).
\end{equation}
We invert this formula to give:
\begin{equation}
\label{minbynew}
\text{Tr}T_n^\chi|\skmin(N,\chi) = \sum_{\langle M,\psi \rangle}(-1)^{k'\left(\frac{N}{M}\right)}2^{-k\left(\frac{N}{M},\chi,\psi \right)}\overline{\psi(n)}\text{Tr}T_n^{\chi \psi^2}|\sknew(M,\chi \psi^2),
\end{equation}
where $k'\left(\frac{N}{M}\right)$ is the number of primes dividing $\frac{N}{M}$.

We now see that multiplicativity results in the trace formula carry through to the twist-minimal trace formula. Let $f(N,\chi)$ be a function defined on $N \in \mathbb{N}$ and $\chi$ level $N$. Fix $n,m \in \mathbb{N}$ and define $f^{\text{new}}$ and $f^{\min}$ as:
\begin{align}
f^{\text{new}}(N,\chi) &= \sum_{\substack{M\mid N \\ \cond(\chi)\mid M}} \ \beta_m(N/M)f(M,\chi)\\
f^{\min}(N,\chi) &= \sum_{\langle M, \psi \rangle}(-1)^{k'(N,M)}2^{-k(N,M,\chi,\psi)}\overline{\psi(n)}f^{\text{new}}(M,\chi\psi^2)
\end{align}
Let $f(N,\chi)$ be multiplicative, so that $f(N,\chi)=\prod_{p\mid N}f(p^{\nu_p(N)},\chi_p)$. Then $f^{\text{new}}$ and $f^{\min}$ are also multiplicative in this way.

We will show that each component of (\ref{traceformulaequation}) can be written as linear combinations of multiplicative functions. Thus, we construct the twist-minimal trace formula by computing each component's value at each prime for arbitrary twist-minimal character $\chi$. This provides an expression of the form:
\begin{equation}
\label{preformula}
\text{Tr}T_n^\chi|\skmin(N,\chi)=\sum_{d \in \mathcal{P}(N,\chi,n)}\chi_{\cond}(d)d^{k-1} (B_1-B_2-B_3+B_4).
\end{equation}
We then perform a further, similar sieve on $n$ to get an expression of the form:
\begin{equation}
\text{Tr}T_n^\chi|\skmin(N,\chi)=C_1-C_2-C_3+C_4,
\end{equation}
as in Theorem \ref{MainTheorem}.

Due to the definition of $\beta_m$ in Lemma \ref{newformequation}, the formulae for $p\mid (N,n)$ and $p\nmid (N,n)$ quickly diverge, and we treat them separately.

\subsection{Formula for $p\nmid (N,n)$ cases}
\label{sieving}

We start by establishing an explicit formula for $f^{\min}$ in terms of $f$ in the case that $\chi$ is twist-minimal.

Let $\underset{C}{\delta}$ be the characteristic function taking 1 when $C$ is true and 0 when $C$ is false. Let $\chi$ be a twist-minimal character of level $p^e$ and conductor $p^s$, as in Definition \ref{tmcdef}. From the definitions of $f^{\text{new}}$ and $f^{\min}$ we get:
\begin{equation}
\label{decompositionformula}
f^{\min}(p^e,\chi) =f(p^e,\chi) + \underset{s<e}{\delta} \left(\begin{split}
&\underset{s \le e-2}{\delta}f(p^{e-2},\chi)-2f(p^{e-1},\chi)\\
&-\frac{1}{2}\underset{\substack{p>2 \\ 2\mid e}}{\delta}\sum_{\substack{\cond(\psi)=p^{\frac{e}{2}}\\ \psi \not = \overline{\chi}}} \overline{\psi (n)} f(p^{e/2},\chi \psi^2)\\
&+  \underset{\substack{p>2 \\ s=0\\ e=2}}{\delta}\left( \frac{n}{p} \right)\left(f(1,\mathbbm{1})-\frac{f(p,\mathbbm{1})}{2}\right)\
\end{split}\right)
\end{equation}
This is what we will refer to as the decomposition formula.

We factor a multiplicative function from each term in (\ref{traceformulaequation}), and then apply the decomposition formula to get expressions for $B_i$ in (\ref{preformula}). In this section, when evaluating the multiplicative function at some prime $p\mid N$, we will assume that $p\nmid n$, with the cofactor case being handled in the next section. Note that this means we can assume $d=1$ in (\ref{preformula}).

At many points it will be useful to have the following lemma.
\begin{lemma}
\label{sumcharformula}
For any prime $p$ and $x,e \in \mathbb{N}$ we have:
\begin{equation}\sum_{\chi,\cond(\chi)=p^a}\chi(x)=\begin{cases}
0 & \text{if }\nu_p(x-1) < a-1 \\
\lfloor p^{a-2} \rfloor - p^{a-1} & \text{if }\nu_p(x-1)=a-1 \\
\lfloor p^{a-2} \rfloor - 2p^{a-1} + p^a & \text{if }\nu_p(x-1) \ge a\\
\end{cases}
\end{equation}
\end{lemma}
\subsubsection{$A_1$ term}
This simple case will be useful in demonstrating the general sieving approach. We let $R$ denote the multiplicative factor of the component (which is different for each component, so any specific $R$ will not persist). In \ref{traceformulaequation} we have already that:
\begin{equation}
A_1=n^{k/2-1}\frac{k-1}{12}\prod_{p\mid N} R(p^e,\chi_p,n)
\end{equation}
where:
\begin{equation}
R(p^e,\chi,n)=\chi(\sqrt{n})\left(p^e+p^{e-1}\right).
\end{equation}
Note that the following expression in the decomposition formula simplifies:
\begin{equation}
\overline{\psi(n)}R(p^e,\chi\psi^2,n)=R(p^e,\chi,n)
\end{equation}
This makes it easy to evaluate $R^{\min}$, which is done on a case-by-case basis for each possible twist-minimal character. Applying Lemma \ref{sumcharformula} we get:
\begin{equation}
R^{\min}(p^e,\chi,n)=\chi(\sqrt{n})\begin{cases}
p^e+p^{e-1} & \text{if }s=e\\
\phi(p^{e-2})\frac{p^2-1}{1+\underset{2\mid e,p>2}{\delta}} & \text{if }s<e, e \ge 3\\
p^2-p-1 & \text{if }s<e, e=2, p=2\\
\frac{p^2-1}{2} & \text{if }s=1,e=2, p>2\\
\frac{(p-1)^2}{2} & \text{if }s=0,e=2, p>2\\
p-1 & \text{if }e=1\\
\end{cases}
\end{equation}
This is then rewritten more succinctly, so that $B_1$ in (\ref{preformula}) is given by:
\begin{equation}
B_1=n^{k/2-1}\chi(\sqrt{n})\frac{k-1}{12}\prod_{p\mid N}\begin{cases}\frac{\phi(\lceil p^{e-2}\rceil)(p-1)}{1+\underset{2\mid e, p>2}{\delta}}(1+\underset{e>1}{\delta}p+\underset{e=2}{\delta}(2s-2)) & \text{if }s<e\\
p^e+p^{e-1} & \text{if }s=e
\end{cases}
\end{equation}
\subsubsection{$A_3$ term}
We will focus on the $A_3$ in (\ref{traceformulaequation}) next, returning to $A_2$ as the most complicated component later. Again we extract a multiplicative factor $R$ which is distinct from the $R$ of the previous section. We write:
\begin{equation}
\sum_{\substack{c\mid N \\ (c,N/c)\mid (N/\cond(\chi),n/d-d)}} \phi((c,N/c))\chi(x_1)=\prod_{p\mid N} R(p^e,\chi_p,d,n),
\end{equation}
where, setting $\gamma=\nu_p(n/d-d)$ we have:
\begin{equation}
R(p^e,\chi,d,n)=(\chi(d)+\chi(n/d))\left(p^{\min(\lfloor \frac{e}{2}\rfloor,e-s,\gamma)}-\underset{\substack{\min(e-s,\gamma)\ge \frac{e}{2}\\ 2\mid e>0}}{\delta}\frac{p^{\frac{e}{2}}-p^{\frac{e}{2}-1}}{2}\right).
\end{equation}
Applying the decomposition formula (\ref{decompositionformula}) to $R$ we see considerable cancellation - although the expression for $R^{\text{new}}$ is quite complicated, the expression for $R^{\min}$ is very simple. We get:
\begin{equation}
R^{\min}(p^e,\chi,d,n)=(\chi(d)+\chi(n/d))\begin{cases}
\frac{\sqrt{2}^e}{8}\left(1-\underset{\gamma=\frac{e}{2}-1}{\delta}2\right) & \text{if }p=2, 2\mid e, \gamma \ge \frac{e}{2}-1,s<\frac{e}{2}\\
1 & \text{if }s=e\\
0 & \text{ otherwise}\\
\end{cases}
\end{equation}
and so $B_3$ is given by:
\begin{equation}
B_3=\sum_{\substack{d\mid n \\ d \le \sqrt{n}}}' d^{k-1}\prod_{p\mid N}(\chi_p(d)+\chi_p(n/d))\begin{cases}
\frac{\sqrt{2}^e}{4}\left(\frac{1}{2}-\underset{\gamma=\frac{e}{2}-1}{\delta}\right) & \text{if }p=2, 2\mid e, \gamma \ge \frac{e}{2}-1,s<\frac{e}{2}\\
1 & \text{if }s=e\\
0 & \text{ otherwise}\\
\end{cases}
\end{equation}
The dash on the sum here holds the same meaning as in the $A_3$ term in (\ref{traceformulaequation}).
\subsubsection{$A_4$ term}
The $A_4$ term is 0 whenever $k>2$. As the weight does not change in the decomposition formula, we can assume that $k=2$. We rewrite $A_4$ as:
\begin{equation}
A_4 = \sigma(n) \prod_{p\mid N} R(\chi_p),
\end{equation}
where:
\begin{equation}
R(\chi_p) = \underset{s=0}{\delta}\begin{cases}
1 & \text{if }p \nmid n\\
\frac{p^{\nu_p(n)}}{\sigma(p^{\nu_p(n)})} & \text{if }p \mid n
\end{cases}
\end{equation}
Note that for this section we are always in the top case. We again sieve $R$ through the decomposition formula, giving:
\begin{equation}
B_4=\underset{\chi=\mathbbm{1}}{\delta}\sigma(n)  \prod_{p\mid N} \begin{cases}
-1 & \text{if }e=1\\
0 & \text{ otherwise}
\end{cases}
\end{equation}
\subsubsection{$A_2$ term}
We follow the work of \cite{BookerLee} to rewrite the $A_2$ term multiplicatively. For $t^2<4n$, write $t^2-4n=d\ell^2$ with fundamental discriminant $d$ and $\ell \in \mathbb{N}$. Note that the sum over $f^2|(t^2-4n)$ is equivalent to taking the sum over $f\mid \ell$, discarding cases where $\nu_2(f)=\nu_2(\ell+1)$ and $4\mid d$, in which the sum over $x \pmod{N}$ is zero.

The class number formula (see \cite{Davenport}) leads to the following result from \cite{BookerLee}:
\begin{equation}
\frac{h(d\ell^2/f^2)}{w(d\ell^2/f^2)}= \frac{lh(d)}{fw(d)}\prod_{p\mid \frac{\ell}{f}}  \frac{p-\left(\frac{d}{p} \right)}{p}.
\end{equation}
Thus, the $A_2$ term becomes:
\begin{equation}
\begin{split}
A_2=&\frac{h(d)}{w(d)}\sum_{f\mid \ell} \frac{l}{f} \sum_{\substack{x \text{ mod }N \\ x^2-tx+n \equiv 0 \text{ (mod }N(N,f))}}\chi(x)\\
&\cdot \prod_{p\mid \ell}p^{\text{min}(e,\nu_p(f))-2}\left(p+\underset{\nu_p(f) \ge e>0}{\delta}\right)\left(p-\underset{\lfloor\frac{\gamma}{2}\rfloor>\nu_p(f)}{\delta}\left(\frac{d}{p}\right)\right)\\
\end{split}
\end{equation}
Let $g$ be some number dividing $(N,\ell)$. Define the function:
\begin{equation}
J(N,g,\chi,t,n) = \begin{cases}
\sum_{\substack{x \text{ mod }N \\ x^2-tx+n \equiv 0 \pmod{Ng}}}\chi(x) & \text{if }N>1\\
1 & \text{ otherwise}
\end{cases}
\end{equation}
Let $\Omega (N,g,t,n)$ be set of solutions mod $N$ to $x^2+tx-n \equiv 0 \pmod{Ng}$ so that $J(N,g,\chi,t,n)=\sum_{x \in \Omega(N,g,t,n)}\chi(x)$. From studying the output of the Chinese Remainder Theorem, we see that we can write $\Omega$ as a Cartesian product:
\begin{equation}
\Omega(N,g,t,n)=\prod_{p\mid N}\Omega(p^{\nu_p(N)},p^{\nu_p(g)},t,n)
\end{equation}
Consequently, $J$ is a multiplicative function. We arrive at the following expression for $A_2$:
\begin{equation}
\begin{split}
A_2=\frac{h(d)}{w(d)}\sum_{f\mid \ell} \frac{l}{f}&\prod_{p\mid N} J(p^e,p^{\min(\nu_p(f),e)},\chi_p,t,n)\\
\cdot&\prod_{p\mid \ell}p^{\text{min}(e,\nu_p(f))-2}\left(p+\underset{\nu_p(f) \ge e>0}{\delta}\right)\left(p-\underset{\lfloor \frac{\gamma}{2} \rfloor>\nu_p(f)}{\delta}\left(\frac{d}{p}\right)\right)
\end{split}
\end{equation}
We rewrite this as one product over $\ell N$, swapping the sum and product to give:
\begin{equation}
A_2=\frac{h(d)}{w(d)}\prod_{p\mid \ell N}R_p(p^e,\chi_p,t,n)
\end{equation}
where:
\begin{equation}
\begin{split}
R_p(p^e,\chi_p,t,n)=&\sum_{j=0}^{\nu_p(\ell)}p^{\min(\nu_p(\ell)+e-j,\nu_p(\ell))-2}\left(p+\underset{j \ge e >0}{\delta} \right)\left(p-\underset{\lfloor \frac{\gamma}{2} \rfloor>j}{\delta}\left(\frac{d}{p}\right)\right)\\
&\cdot J(p^e,p^{\min(j,e)},\chi_p,t,n)
\end{split}
\end{equation}
Sieving over $N$ we get:
\begin{equation}
B_2 = \frac{h(d)}{w(d)}\prod_{\substack{p\mid \ell \\ p\nmid N}}R_p(1,\mathbbm{1},t,n)\prod_{p\mid N}R_p^{\min}(p^e,\chi_p,t,n),
\end{equation}
where:
\begin{equation}
\label{Sp1}
R_p(1,\mathbbm{1},t,n) = p^{\nu_p(\ell)}+\left(1-\left(\frac{d}{p}\right)\right)\frac{p^{\nu_p(\ell)}-1}{p-1}.
\end{equation}
We now find an explicit expression for $R_p^{\min}$ by sieving. This first requires an explicit expression for $R_p$ in the case that $\chi$ is twist-minimal. From now on, set $\gamma = \nu_p(t^2-4n)$ with prime $p$ being clear from context. Also, for any $a$ such that $(a,\cond(\chi))=1$, we let $\chi(1/a)$ mean $\chi(b)$ where $ab \equiv 1 \pmod{\cond(\chi)}$.

We begin by finding explicit expressions for the $J$ function.
\begin{lemma}
Let $\chi_p$ be a character level $p^e$. Let $\beta \in \mathbb{Z}_{\ge 0}$ be such that $\beta \le e$ and $\beta \le \nu_p(\ell)$ where $t^2-4n=d\ell^2$ with fundamental discriminant $d$. For any prime $p>2$ we get:
\begin{equation}
J(p^e,p^\beta,\chi_p,t,n)= \begin{cases}
\chi(t/2)p^{\lfloor \frac{e-\beta}{2} \rfloor}&\text{if }\substack{\cond(\chi) \le p^{\lceil \frac{e+\beta}{2} \rceil}\\\gamma \ge e+\beta}\\
\\
p^{\gamma/2-\beta}(\chi(t/2+up^{\gamma/2}/2) & \text{if }\substack{\left(\frac{d}{p}\right)=1 \\ \gamma <e+\beta , \cond(\chi) \le p^{e+\beta - \gamma/2}}\\
\; \; \; \; +\chi(t/2-up^{\gamma/2}/2))\\
\\
0 & \text{else}
\end{cases}
\end{equation}
where $u$ is any possible value satisfying $u^2 \equiv \frac{t^2-4n}{p^\gamma} \pmod{p^{e+\beta-\gamma}}$. For $p=2$, if $2\nmid t$ then $J=0$. Assuming $2\mid t$ we get:
\begin{equation}
J(2^e,2^\beta,\chi_2t,n)= \begin{cases}
\chi(t/2)2^{\lfloor \frac{e-\beta}{2} \rfloor}&\text{if }\substack{\cond(\chi) \le 2^{\lceil \frac{e+\beta}{2} \rceil}\\ \gamma \ge e+\beta+2}\\
\\
2^{\gamma/2-\beta-1}\chi(t/2+2^{\gamma/2-1}) & \text{if }\substack{d\equiv1,4,5\mod 8,\; e+\beta+2-\gamma=1 \\ \cond(\chi) \le 2^{\gamma/2}}\\
\\
2^{\gamma/2-\beta}(\chi(t/2+2^{\gamma/2-1})& \text{if }\substack{d \equiv 1,5 \mod 8,\; e+\beta+2-\gamma=2 \\ \cond(\chi) \le 2^{\gamma/2}}\\
\\
2^{\gamma/2-\beta+1}(\chi(t/2+u2^{\gamma/2-1}) & \text{if }\substack{d \equiv 1 \mod 8,\;e+\beta+2-\gamma \ge 3 \\ \cond(\chi) \le 2^{\gamma/2}}\\
\\
2^{\gamma/2-\beta}(\chi(t/2+u2^{\gamma/2-1})\\
\; \; \; \; \; \; +\chi(t/2-u2^{\gamma/2-1})) & \text{if }\substack{d \equiv 1 \mod 8,\;e+\beta+2-\gamma \ge 3 \\ 2^{\gamma/2+1} < \cond(\chi) \le 2^{e+\beta-\gamma/2}}\\
\\
0 & \text{ else }
\end{cases}
\end{equation}
\end{lemma}
\begin{proof}
First, assume that $p>2$. If $\gamma \ge e+\beta$ then $t^2-4n \equiv 0 \text{ (mod }p^{e+\beta})$ and so we need $2x \equiv t \text{ (mod }p^{\lceil \frac{e+\beta}{2} \rceil})$. This gives $p^{\lfloor \frac{e-\beta}{2}\rfloor}$ elements of $\Omega(p^e,p^\beta,t,n)$. If $\cond(\chi) \le p^{\lceil \frac{e-\beta}{2} \rceil}$ then $\chi(x)$ is equal for all $x \in \Omega$. If, however $\cond(\chi) > p^{\lceil \frac{e+\beta}{2} \rceil}$ then the sum is 0 by orthogonality.

If instead we have $\gamma < e+\beta$, then $t^2-4n \not \equiv 0 \text{ (mod }p^{e+\beta})$. If $2\nmid \gamma$ then $\Omega=\emptyset$. If $2\mid \gamma$ then $\frac{t^2-4n}{p^\gamma}$ has 2 roots mod $p^{e+\beta-\gamma}$. Denote one of them $u$, so that $\Omega(p^e,p^\beta,t,n)$ is all elements satisfying $2x \equiv t\pm up^{\gamma/2}$ mod $p^{e+\beta-\gamma/2}$. If $\cond(\chi)>p^{e+\beta-\gamma/2}$ then the sum is 0 by orthogonality, else we get $p^{\gamma/2-\beta}(\chi(t/2+up^{\gamma/2}/2)+\chi(t/2-up^{\gamma/2}/2))$. 
Finally, note that the two conditions $2\mid \gamma$ and $\left(\frac{(t^2-4n)/p^\gamma}{p} \right)=1$ can be combined as $\left( \frac{d}{p} \right)=1$.

Now, assume that $p=2$. If $2\nmid t$ then $x^2-tx+n \equiv x(x-1)+1 \equiv 1 \text{ (mod }2)$ so $\Omega=\emptyset$. Instead, assume $2\mid t$, and so $\chi(t/2)$ is well defined. When needed, write $t=2t'$.

If $\gamma \ge  e+\beta+2$ then $\Omega$ is $2^{\lfloor \frac{e-\beta}{2}\rfloor}$ elements, of the form $t'+kp^{\lceil \frac{e+\beta}{2} \rceil}$ where $k$ ranges through $[1,2^{\lfloor \frac{e-\beta}{2} \rfloor}]$.

If instead we have $\gamma < e+\beta$ then $2\nmid \gamma$ means $\Omega=\emptyset$. Otherwise, let $u_i$ be any root of $\frac{t^2-4n}{2^{\gamma}}$ mod $p^{e+\beta+2-\gamma}$. Elements in $\Omega$ are of the form $x=t'+u_i2^{\gamma/2-1}+k2^{e+\beta+1-\gamma/2}$. The result follows from summation over these elements, simplifying by appeals to orthogonality.\end{proof}
Now that we have expressions for $J$ we move on to evaluating $R$. We use the following formula, arrived at through application of the standard geometric series sum.
\begin{equation}
\label{geometricfloorformula}
\sum_{j=a}^b p^{\lfloor \frac{c-j}{2} \rfloor} = p^{c-\lfloor \frac{c+a}{2} \rfloor} + p^{c - \lfloor \frac{c+a-1}{2} \rfloor} -p^{\lfloor \frac{c-b}{2} \rfloor} - p^{\lfloor \frac{c-b+1}{2} \rfloor}
\end{equation}
Let $p>2$. In the sum defining $R$ we study the cases when $e<\nu_p(\ell)$, $\nu_p(\ell)<e\le \gamma$ and $\gamma < e$ separately. In each case, the resultant sum is evaluated using (\ref{geometricfloorformula}). We end up with the following summary of all cases. Set $h=\max(2s-1,e)$, then we get:
\begin{equation}
\begin{split}
&R_p(p^e,\chi_p,t,n)=\\
&\underset{\gamma \ge h }{\delta}\left(p^{e+\nu_p(l)-1}(p^{-\lfloor \frac{h}{2} \rfloor}+p^{-\lfloor \frac{h-1}{2} \rfloor})+\left( 1- \left(\frac{d}{p}\right)\right) \frac{p^{e-1}}{p-1}(p^{\nu_p(l)-\lfloor \frac{h}{2} \rfloor}+p^{\nu_p(l)-\lfloor \frac{h-1}{2} \rfloor}-p-1)\right)\chi(t/2) \\
&+\underset{\substack{\gamma<h\\ \left(\frac{d}{p}\right)=1}}{\delta}p^{\nu_p(l)+\text{min}(e-s,\nu_p(l))}\left(\chi(t/2+up^{\nu_p(l)})+\chi(t/2-up^{\nu_p(l)}))\right).
\end{split}
\end{equation}
This matches the formula in \cite[Equation 2.41]{BookerLeeStrombergsson}.

Recall that $\gamma=\nu_p(t^2-4n)$ and $h=\max(2s-1,e)$. When $p=2$ we split based on the relation between these quantities. In each case, we get a different result based on the congruence class of $d$ mod 8. Again, applying (\ref{geometricfloorformula}), we arrive at the following expression for all cases:
\begin{equation}
\begin{split}
&R_2(2^e,\chi,t,n)=\\
&\begin{cases}
\underset{d \equiv 1}{\delta}2^{\nu_2(\ell)+\min(e-s,\nu_2(\ell))}(\chi(t/2+u2^{\gamma/2-1})+\chi(t/2-u2^{\gamma/2-1})) & \text{if }\gamma<h\\
-3 \cdot 2^{e-1}\chi(t/2) & \text{if }\gamma=h=2s, 2\nmid d\\
3 \cdot 2^{e-1}\chi(t/2) & \text{if }\gamma=h>2s, 2\nmid d\\
-\left(\underset{d \not \equiv 4}{\delta} 2^{e+1} + \left(1-\left(\frac{d}{2}\right)\right)2^{e-1}\right)\chi(t/2)&\text{if }\gamma=2s>h\\
\\
\chi(t/2)\cdot \underset{2 \nmid d}{\delta} 2^{e-1}(2^{\nu_2(\ell)-\lfloor \frac{h}{2} \rfloor}+2^{\nu_2(\ell)-\lfloor \frac{h-1}{2} \rfloor})&\text{if }\gamma > \max(h,2s)\\
+\underset{\gamma>e+1 \text{ or }d\not \equiv 8}{\delta}\left(1-\left(\frac{d}{2}\right)\right)2^{e-1}(2^{\lfloor \frac{\gamma}{2} \rfloor - \lfloor \frac{h}{2} \rfloor}+2^{\lfloor \frac{\gamma}{2}\rfloor- \lfloor\frac{h-1}{2} \rfloor} - 3)\\
\end{cases}
\end{split}
\end{equation}
With expressions for $R_p$, this is now sieved using the decomposition formula (\ref{decompositionformula}). We treat the odd and even prime cases separately, and sieve to $R_p^{\text{new}}$ first, given by:
\begin{equation}
R_p^{\text{new}}(p^e,\chi,t,n)=R_p(p^e,\chi,t,n)-2R_p(p^{e-1},\chi,t,n)+\underset{s \le e-2}{\delta}R_p(p^{e-2},\chi,t,n)
\end{equation}
Let $p$ be an odd prime. We study the cases $e=1,e=2$ and $e\ge 3$ separately. The resultant expression for $R_p^{\text{new}}(p^e,\chi_p,t,n)$ is given by:
\begin{equation}
\begin{split}
R_p^{\text{new}}=
&\chi(t/2)\begin{cases}
p^{e-3}\left(\underset{2\mid e}{\delta}\left(p-\left(\frac{d}{p}\right)\right)p^{\nu_p(l)-\frac{e}{2}+1}(p-1)-\left(1-\left(\frac{d}{p}\right)\right)(p^2-1)\right) & \text{if }\gamma \ge e-1 \ge 2\\
p^{e-3}\left(1-p\left(\frac{d}{p}\right)\right) & \text{if }\gamma=e-2 \ge 1\\
-\left(1-\left(\frac{d}{p}\right)\right) & \text{if }e=1\\
p^{\nu_p(l)}(p-2-s)+\left(1-\left(\frac{d}{p}\right)\right)\frac{p^{\nu_p(l)}(p-2-s)-p^2+p+1+s}{p-1} & \text{if }\gamma \ge e=2\\
-1-s & \text{if } \gamma=1=e-1\\
0 & \text{otherwise}
\end{cases}\\
&-\underset{e=2,\gamma=0}{\delta}\begin{cases}
\chi(t/2+u)+\chi(t/2-u) & \text{if }s=1, \left(\frac{d}{p}\right)=1\\
\left(\frac{d}{p}\right) & \text{if }s=0\\
0 & \text{ otherwise}
\end{cases}\\
\end{split}
\end{equation}
Assuming, now, that $2\mid e$ we have a further sieve for $R_p^{\min}$. 
This is only considered when $p>2$ and $2\mid e$. We start by working out the sum over primitive characters. Note that we always have $\cond(\chi\psi^2)=p^{\frac{e}{2}}$ and so $h=e-1$. Using this, along with Lemma \ref{sumcharformula}, we get the following when $\gamma \ge e-1$: 
\begin{equation}
\begin{split}
\sum_{\psi,\cond(\psi)=p^{\frac{e}{2}}}\overline{\psi(n)}R_p(p^{\frac{e}{2}},\chi\psi^2,t,n)=&p^{e-3}(p-1)((p-1)2p^{\nu_p(l)-\frac{e}{2}+1}+\left(1-\left(\frac{d}{p}\right)\right)\\
&\cdot(2p^{\nu_p(l)-\frac{e}{2}+1}-p-1))\chi(t/2)\\
\end{split}
\end{equation}
When $\gamma < e-1$ and $\left(\frac{d}{p}\right)=1$ we get:
\begin{equation}
\sum_{\psi,\cond(\psi)=p^{\frac{e}{2}}}\overline{\psi(n)}R_p(p^{\frac{e}{2}},\chi\psi^2,t,n)=\begin{cases}
-2\phi(p^\gamma)\chi(t/2) & \text{if }\gamma=e-2\\
0 & \text{ otherwise}
\end{cases}
\end{equation}
We use these to find $R_p^{\min}$ when $e>2$, and then treat the $e=2$ case separately. In full, for any twist-minimal $\chi$ with $s<e$ and $p>2$ we get $R_p^{\min}(p^e,\chi,t,n)$ given by:
\begin{equation}
\underset{\substack{e=1 \\ \text{ or } \left(\frac{n}{p}\right)=1}}{\delta} \underset{\gamma \ge e-2}{\delta}\left(1-\left(\frac{d}{p}\right)\right)\frac{p^{e-3}}{(2,e)}\chi(t/2)\left(\underset{e>2}{\delta}+p\left(\underset{e=2}{\delta}(1-2s)+\underset{\substack{2\mid e \\ \gamma=e-2}}{\delta}-\underset{\gamma \ge e-1}{\delta}p\right)\right)
\end{equation}
When $p=2$ there are more character cases to consider. They are studied in the following categorisation, which is seen to be exhaustive:
\begin{itemize}
\item $s=\lfloor \frac{e}{2} \rfloor, e \ge 4$
\item $s \in \{0,2\}, e \ge \max(2s+2,3)$
\item $e < 3$
\end{itemize}
For each case, we further separate cases on the value of $\gamma$ and its relation to $e$. In full, we get:
\begin{equation}
R_p^{\min}=\left(1-\left(\frac{d}{2}\right)\right)\chi(t/2)\lceil 2^{e-3}\rceil \begin{cases}
-3 & \text{if }\gamma > e, e \ge 3\\
(-1)^e+2 & \text{if }\gamma=e,s=\lfloor \frac{e}{2}\rfloor, e \ge 4 \\
2(-1)^{d+1}+1 & \text{if }\gamma=e-1,2\nmid e, s=\lfloor \frac{e}{2} \rfloor, e \ge 4\\
2(-1)^{d}-1 & \text{if }\gamma \in \{e,e-1\}, s<\lfloor \frac{e}{2} \rfloor, e \ge 3\\
-1 & \text{if }\gamma \ge e, e \in \{1,2\}\\
\frac{1}{2} & \text{if }e=2, \gamma=0\\
-1 & \text{if }e=1, \gamma=0\\
\end{cases}
\end{equation}
When $s=e$ there is no sieving, and so we just apply this special case to the original expressions for $R$. When $p>2$ we get:
\begin{equation}
R_p^{\min}=\begin{cases}
\chi_p(t/2)\left(2p^{\nu_p(\ell)}+\left(1-\left(\frac{d}{p}\right)\right)\frac{2p^{\nu_p(\ell)}-p^e-p^{e-1}}{p-1}\right) & \text{if }s=e,\gamma \ge 2e-1 \\ \\\underset{\left(\frac{d}{p}\right)=1}{\delta}p^{\nu_p(\ell)}(\chi_p(\frac{t+u}{2})+\chi_p(\frac{t-u}{2})) & \text{if }s=e,\gamma<2e-1\\
\text{where }u\equiv l\sqrt{d} \pmod{p^{e+\nu_p(\ell)}} \\ 
\end{cases}
\end{equation}
and when $p=2$ we get:
\begin{equation}
R_p^{\min}=\begin{cases}
\left((2^{\lfloor \frac{\gamma}{2}\rfloor+1}-3\cdot2^{e-1})(1-\left(\frac{d}{2}\right))+\underset{2\nmid d}{\delta}2^{\nu_2(\ell)+1}\right)\chi_2(t/2) & \text{if }\gamma>2e\\ \\
-\left(2^{e-1}\left(1-\left(\frac{d}{2}\right)\right)+\underset{2\nmid d}{\delta}2^{\nu_2(\ell)+1}\right)\chi_2(t/2) & \text{if }\gamma=2e\\ \\
2^{\nu_2(\ell)}(\chi_2\left(\frac{t+u}{2}\right)+\chi_2\left(\frac{t-u}{2}\right)) & \text{if }\gamma<2e-1,\left(\frac{d}{2}\right)=1\\
\text{where }u\equiv l\sqrt{d} \pmod{p^{e+\nu_p(\ell)+2}} \\ 
0 & \text{otherwise}
\end{cases}
\end{equation}
\subsection{Formula for $p\mid (N,n)$}
We now address the case that $p$ divides both $n$ and $N$, finding expressions for the multiplicative part of each term (which we have denoted $R$ in each case). From Lemma \ref{newformequation} we see that if $s<e$ and $e>1$ then the trace is 0. As such, we only consider terms when $s=e$ or $e=1$. Noting that $\beta_m$ is now equivalent to the M\"{o}bius $\mu$ function, we get the much simpler decomposition formula:
\begin{equation}
\label{decompositionformulacofactor}
f^{\min}(p^e,\chi) =f(p^e,\chi) - \underset{s<e}{\delta} f(1,\mathbbm{1})
\end{equation}

For the $A_1$ term, if $s>0$ then as $p\mid n$ we necessarily have $\chi(\sqrt{n})=0$. Thus, we only have a non-zero $A_1$ term when $e=1$ and $s=0$. In this case, we get $R_p^{\min}=-1$ and so in full:
\begin{equation}
R_p^{\min}(p,\chi,n)=-\underset{s=0}{\delta}
\end{equation}
For the $A_2$ term, depending on $\gamma$, we have at most two elements in $\Omega$. If $\gamma=\nu_p(t^2-4n)$ is non-zero, then we must have $p\mid x$ whereupon $\chi(x)=0$. On the other hand, if $\gamma=0$ then we get $J=\chi_p\left(\frac{t+u}{2}\right)+\chi_p\left(\frac{t-u}{2}\right)$. We get:
\begin{equation}
R_p^{\min}(p^e,\chi,t,n)=\begin{cases}
-\left(p^{\nu_p(\ell)}+\left(1-\left(\frac{d}{p}\right)\right)\frac{p^{\nu_p(\ell)}-1}{p-1}\right) & \text{if }s=0,\gamma>0\\
\chi_p\left(\frac{t-u}{2}\right)+\chi_p\left(\frac{t+u}{2}\right) & \text{if }s=e, \gamma=0 \\
0 & \text{ otherwise}
\end{cases}
\end{equation}
For the $A_3$ term note that $\chi(x_1)\not =0$ for either $c=1$ or $c=p$ if and only if $p\nmid d$ or $p \nmid n/d$, which is subsequently true if and only if $\nu_p(n/d-d)=0$. Using this to evaluate, we get:
\begin{equation}
R_p^{\min}(p^e,\chi,d,n)=\chi(d)+\chi(n/d)-\underset{s=0}{\delta}
\end{equation}
Finally for the $A_4$ term we get:
\begin{equation}
R_p^{\min}(p,n)=\underset{s=0}{\delta}\left(\frac{p^{\nu_p(n)}}{\sigma(p^{\nu_p(n)})}-1\right)
\end{equation}
\subsection{Factorising over $n$}
The preceding subsections allow us to produce a formula in the form of (\ref{preformula}),
where each $B_i$ is evaluated on the arguments $(\frac{N}{d},\chi,k,\frac{n}{d^2})$. We now incorporate the outer sum by swapping the sum and product in each term.

For the $B_1$ term, swapping gives:
\begin{equation}
\begin{split}
C_1=&\prod_{\substack{p\mid N \\ p \nmid n}}\begin{cases}
p^e+p^{e-1}& \text{if }s=e\\
\frac{\phi (\lceil p^{e-2}\rceil)(p-1)}{1+\underset{2\mid e,p>2}{\delta}}(1+\underset{e>1}{\delta}p+\underset{e=2}{\delta}(2s-2))&\text{if }s<e
\end{cases}\\
\cdot & \sum_{d \in \mathcal{P}}d \frac{n^{k/2-1}(k-1)\chi_{\cond}(\sqrt{n})}{12}
\prod_{\substack{p\mid \frac{N}{d}\\ p\mid \frac{n}{d^2}}}(-1)
\end{split}
\end{equation}
Then, using multiplicativity of $\mathcal{P}$ as defined in (\ref{mathcalp}), we evaluate the inner sum, giving:
\begin{equation}
C_1=\frac{n^{k/2-1}(k-1)\chi_{\cond}(\sqrt{n})}{12}\prod_{p\mid N}\begin{cases}
p^e+p^{e-1}& \text{if }s=e\\
\frac{\phi (\lceil p^{e-2}\rceil)(p-1)}{1+\underset{2\mid e,p>2}{\delta}}(1+\underset{e>1}{\delta}p+\underset{e=2}{\delta}(2s-2))&\text{if }s<e\\
\end{cases}
\end{equation}
For $B_3$ swapping the order of summation and substituting $md$ for $m$ gives:
\begin{equation}
\begin{split}
C_3 = &\sum_{\substack{m\mid n \\ m\le \sqrt{n}}}' m^{k-1} \sum_{\substack{d\in \mathcal{P}(N,n,\chi)\\ d\mid (m,n/m)}}\\
\cdot&\prod_{p\mid \frac{N}{d}}\begin{cases}
\frac{\sqrt{2}^e \left(\chi_p(m)+\chi_p\left(n/m\right)\right)}{8}(1-\underset{\gamma=\frac{e}{2}-1}{\delta}2) & \text{if }p=2, 2\mid e, \gamma \ge \frac{e}{2}-1 \ge s, 2\nmid \frac{n}{d^2}, e>2\\
\chi_p(m)+\chi_p\left(n/m\right)-\underset{s=0}{\delta} & \text{if }s=e \text{ or }p\mid \frac{n}{d^2}\\
0 & \text{ otherwise}\\
\end{cases}\\
\end{split}
\end{equation}
Noting that if $p\nmid n$ then the product evaluates identically for all $d$, and using multiplicativity of $\mathcal{P}$ we get:
\begin{equation}
C_3=\sum_{\substack{d\mid n \\ d \le \sqrt{n}}}' d^{k-1}\prod_{p\mid N}\begin{cases}
\frac{\sqrt{2}^e (\chi_p(d)+\chi_p(n/d))}{8}(1-\underset{\gamma=\frac{e}{2}-1}{\delta}2) & \text{if }p=2, 2\mid e, \gamma \ge \frac{e}{2}-1 \ge s, 2\nmid n, e>2\\
\chi_p(d)+\chi_p(n/d) & \text{if }s=e\\
0 & \text{ otherwise}\\
\end{cases}\\
\end{equation}
For the $B_4$ contribution we get:
\begin{equation}
\begin{split}
C_4=&\underset{k=2,\chi=\mathbbm{1}}{\delta}\prod_{\substack{p\mid N \\ p\nmid n}} \mu(p^e) \sum_{d \in \mathcal{P}(N,n,\chi)} d\sigma(n/d^2)\prod_{p\mid \left(\frac{N}{d},\frac{n}{d^2}\right)}\left(\frac{p^{\nu_p(n/d^2)}}{\sigma(p^{\nu_p(n/d^2)})}-1\right)\\
=&\underset{k=2, \chi=\mathbbm{1}}{\delta}\mu(N)\prod_{\substack{p \mid n \\ p \nmid N}}\sigma(p^{\nu_p(n)})
\end{split}
\end{equation}
For the $B_2$ term, we follow the same procedure as above. Noting that $\gamma$ is unchanged by substituting $n/d^2$ and $t/d$ for $n$ and $t$ respectively, we get:
\begin{equation}
R_p^{\min}(p^e,\chi,t,n)=\begin{cases}
\left(\frac{d}{p}\right)-1 & \text{if }\gamma>s=0\\
\chi_p\left(\frac{t-u}{2}\right)+\chi_p\left(\frac{t+u}{2}\right) & \text{if }s=e, \gamma=0 \\
0 & \text{ otherwise}
\end{cases}
\end{equation}
These $C_i$ terms are now exactly as in Theorem \ref{MainTheorem}.
\pagebreak
\bibliography{papers}
\bibliographystyle{plain}

\end{document}